\documentclass[11pt]{amsart}
\usepackage[all]{xy}
\usepackage{nicefrac}
\usepackage{amsmath}

\swapnumbers

\theoremstyle{plain}
\newtheorem{prop}[subsection]{Proposition}
\newtheorem{sublemm}[subsubsection]{Lemma}
\newtheorem{lemm}[subsection]{Lemma}
\newtheorem{theo}[subsection]{Theorem}
\newtheorem{observ}[subsection]{Observation}

\def\x{\times}
\def\t{\otimes}
\def\fdr{\rightarrow}

\def\PP{\mathcal P}
\def\KK{\mathbb K}
\def\NN{\mathbb N}

\DeclareMathOperator{\coker}{Coker}
\DeclareMathOperator{\Span}{Span}

\title[PBW criterion for operads]{A Poincar\'e-Birkhoff-Witt criterion \\ for Koszul operads}
\author{Eric Hoffbeck}
\address{Laboratoire Paul Painlev\'e, Cit\'e Scientifique - B\^atiment M2, Universit\'e de Lille 1,
59655 Villeneuve d'Ascq Cedex, France}
\email{Eric.Hoffbeck@math.univ-lille1.fr}

\begin{document}

\begin{abstract}
The aim of this article is to give a criterion, generalizing the criterion introduced by Priddy for algebras, to verify that an operad is Koszul. We define the notion of a Poincar\'e-Birkhoff-Witt basis in the context of operads. Then we show that an operad having a Poincar\'e-Birkhoff-Witt basis is Koszul. Besides, we obtain that the Koszul dual operad has also a Poincar\'e-Birkhoff-Witt basis.

We check that the classical examples of Koszul operads (commutative, associative, Lie) have a Poincar\'e-Birkhoff-Witt basis. 
\end{abstract}
    
\maketitle

\setcounter{tocdepth}{1} \tableofcontents

The notion of an operad is used to model categories of algebras.
An appropriate (co)homologicy theory is associated to each category of algebras associated to an operad.
The Koszul duality of operads, introduced by Ginzburg and Kapranov \cite{GK}, allows us to understand the structure
of the (co)homologicy theory associated to some operads, the Koszul operads :
When an operad is Koszul, we know exactly the multiplicative structure of the associated (co)homology
and we have an explicit complex which allows us to determine practically the (co)homology of an algebra. 
Usual examples include the Hochschild complex for the associative operad $\mathcal As$,
the Chevalley-Eilenberg complex for the Lie operad $\mathcal Lie$,
the Harrison complex for the commutative operad $\mathcal Com$.

We refer the reader to \cite{Loday, MSS} for a comprehensive introduction to operads and to \cite{Benoit, GK} for the theory of Koszul operads.

\vspace{0.3cm}

The aim of this article is to give a criterion, generalizing the criterion introduced by Priddy for algebras in \cite{Prid}, to show that an operad is Koszul. This criterion relies on the existence of a basis, called Poincar\'e-Birkhoff-Witt (for short, we write PBW) basis, together with a suitable ordering.

In the case of an algebra $A=\KK<x_1, \ldots, x_n>/I$, a PBW basis consists of a set of monomial representatives of a basis of $A$ so that the product of basis elements remains in the basis or reduces to a sum of larger (for an appropriate order) elements in the basis. A PBW algebra is an algebra equipped with such a basis. Priddy's criterion asserts that a PBW algebra is Koszul. In the context of operads, we replace monomials by treewise compositions of generating operations and we adapt the order property. In this fashion, we generalise Priddy's definition to have an appropriate notion of a PBW basis for operads, and we prove that a PBW operad is Koszul. This gives an answer to a question asked by Kriz  in his review \cite{Kriz} of the article of Ginzburg and Kapranov \cite {GK}. In the article, we prove also that the Koszul dual of a PBW operad is a PBW operad as well. Then we shall see that many usual operads (including commutative, Lie, associative) are PBW.

\vspace{0.3cm}

In sections 1-2, we recall the definitions of the operadic bar construction and conventions on trees used in the definition of an operadic PBW basis. In sections 3-4, we define the notion of a PBW operad and we prove the criterion, such an operad is Koszul. In section 5, we prove that the Koszul dual operad of a PBW operad is PBW as well, with an explicit basis. In section 6, we address the case of non-symmetric operads. To conclude the paper, we examine applications of the criterion in some examples. 

\vspace{0.3cm}

\textbf{Conventions.} We are given a ground field $\KK$, fixed once and for all, of any characteristic. We deal with differential lower graded modules over $\KK$ (for short dg-modules). We only consider operads $\PP$ equipped with a trivial differential, but possibly with a grading. In this context, the usual sign rule applies to the elements of $\PP$. A non-graded operad can be viewed as a graded operad concentrated in degree $0$. 

\textbf{Remark.} Our theorems remain true if $\KK$ is a ring, provided that we restrict ourselves to objects formed from free $\KK$-modules. But the basis conditions are in this context more difficult to verify.

\section{Bar construction and Koszul duality for operads}
In this section, we recall the definition of the reduced bar construction and the definition of Koszul duality for operads. For more details and references, we refer the reader to \cite{Benoit}.

\subsection{Augmentation ideal of an operad}
The {\it identity operad} is defined by $I(r)=\KK$ for $r=1$ and $I(r)=0$ for $r \neq 1$. 
An operad $\PP$ equipped with a morphism $\epsilon : \PP \fdr I$ is called an augmented operad. The augmentation ideal of $\PP$ is $\tilde{\PP}= \ker \epsilon$. As $\epsilon$ is a retract of the identity morphism, we have a splitting $\PP = I \oplus  \tilde{\PP}$.

\subsection{Reduced bar construction}
Recall that the {\it suspension of a dg-module} $M$ is the dg-module $\Sigma M$ defined by $\KK e \t M$, where $deg(e)=1$. We have a natural identification $(\Sigma M)_d = M_{d-1}$.
For a non graded operad $\PP$, the module $\Sigma \tilde{\PP}(r)$ is equal to the module $\tilde{\PP}(r)$ in degree 1 and is zero in degree $* \neq 1$.

The {\it reduced bar construction} $B(\PP)$ is a quasi-cofree cooperad defined by $F^c(\Sigma \tilde{\PP})$, the cofree cooperad generated by the suspension $\tilde{\PP}$. The bar construction $B(\PP)$ is equipped with a differential is given by a coderivation $\partial : F^c(\Sigma \tilde{\PP}) \fdr F^c(\Sigma \tilde{\PP})$ which is determined by the partial composition products of $\PP$.
Recall that $F^c(\Sigma \tilde{\PP})$ is generated by tensors $\bigotimes_{v\in V(\tau)} x_v$, where $\tau$ ranges over trees, the notation $V(\tau)$ refers to the set of vertices of $\tau$ and $x_v$ is an element of $\Sigma \tilde{\PP}$ associated to each vertex. More details on this construction are given in section \ref{bardetails}.

\vspace{0.3cm}

We are interested in the homology of this bar complex. To calculate it, we use that many operads come equipped with a weight grading.

\subsection{Modules equipped with a weight grading}\label{poidstens}
We consider $\KK$-modules  $V$ equipped with a weight grading, a splitting $V = \bigoplus V_{(s)}$.
In the case of a dg-module $V$, the homogeneous components $V_{(s)}$ are supposed to be sub-dg-modules of $V$.
A tensor product of modules equipped with a weight grading inherits a natural weight grading such that 
$(V \t W) _{(n)} = \bigoplus_{s+t=n} V_{(s)} \t W_{(t)}$. 

\subsection{Operads equipped with a weight grading}\label{poidsoperadelib}
An operad $\PP$ is {\it equipped with a weight grading} if each term $\PP(n)$ is weight graded and the composition product
$\PP \circ \PP \fdr \PP$ preserves the weight grading.
This condition asserts equivalently that the partial composition product
of homogeneous elements $p \in \PP_{(s)}(m)$ and $ q  \in \PP_{(t)}(n)$
verify $p \circ_i q \in \PP_{(s+t)} (m+n-1)$.

An operad equipped with a weight grading is called {\it connected} if
$$\PP_{(0)}(r) = \left \{ \begin{array}{rl} \KK . 1 & \text{for $r=1$} \\ 0 & \text{else} \end{array} \right. $$
A connected operad is automatically augmented, the augmentation being the projection on the weight $0$ component. We have $\tilde{\PP}_{(s)}=\PP_{(s)}$ if $s \neq 0$.

In what follows, we will use the free operad $F(M)$. This operad has a natural weight which makes it a graded operad. Recall briefly that 
the free operad $F(M)$, like the cofree cooperad $F^c(M)$, is generated by tensors on trees $\bigotimes_{v\in V(\tau)} x_v$, representing formal compositions of operations.
The weight of such a tensor in $F(M)$ is given by its number of factors $x_v$.
Notice that the free operad is connected. We will go back to the construction of $F(M)$ in section \ref{operadelibre}.

\subsection{Homogeneous operadic ideals and quotients}

A {\it homogeneous operadic ideal} is an operadic ideal $I$ such that $I= \bigoplus I_{(s)}$, where 
$I_{(s)} = I \cap \PP_{(s)}$.

We observe that the quotient of an operad equipped with a weight by a homogeneous ideal is equipped naturally with a weight.
This assertion is an obvious generalization of a classical result for algebras.

\subsection{Quadratic operads}
A {\it quadratic operad} is an operad such that $\PP = F (M)/I$, where $I=(\overline{R})$ is the operadic ideal generated by $R \subset F_{(2)}(M)$.

For the Koszul duality, we use $\overline{R} \subset F_{(2)}(M)$ the sub-$\Sigma_*$-module generated by $R \subset F_{(2)}(M)$. This sub-$\Sigma_*$-module generates the same operadic ideal $(R) = (\overline{R})$.

We will see that the elements of $(\overline{R})$ are represented by trees where one of the vertices is labelled by an element of $\overline{R}$ and the other vertices by elements of $M$.

A quadratic operad has a natural weight grading, induced by the weight grading of the free operad.

For a quadratic operad such that $M(0)=0$, we have automatically 
$$\PP_{(0)}(r) = \left \{ \begin{array}{rl} \KK . 1, & \text{if $r=1,$} \\ 0, & \text{otherwise.} \end{array} \right. $$ 
We have a natural isomorphism $\displaystyle{\PP_{(1)}(r) = M(r)}$. Moreover, we have $\PP_{(2)}(r) = \mathcal F_{(2)} (M)/\overline{R}$. 
 
The operads associated respectively to the associative, commutative, and Lie algebras are quadratic.

\subsection{Weight grading on the bar construction}
If $\PP$ is equipped with a weight grading, then $B(\PP)$ has an induced weight grading. Formally, we use that $B(\PP)$ is spanned by tensors $\bigotimes_v p_v$. The weight of such a tensor is the sum of the weight of the factors $p_v$, as defined in section \ref {poidstens}.
The differential is homogeneous.

If we suppose that $\PP_{(0)}$ is reduced to $\KK.1$, then $\Sigma \tilde{\PP}_{(0)}=0$. Hence the elements $p_i$ which occur in the treewise tensors of $B(\PP)$ have a weight larger than $1$. As a consequence, we have $B_d(\PP)_{(s)}=0$ if $d>s$.

\subsection{Koszul operads}\label{Koszul}
We work with the definition given by Fresse in \cite{Benoit}. It generalizes the original definition by Ginzburg and Kapranov in \cite{GK} for operads with are not generated by binary operations.

Ones says that a (connected, graded, equipped with weight) operad $\PP$ is Koszul if $H_*(B_*(\PP)_{(s)})=0$ for $* \neq s$ (in words if the homology of its bar construction is concentrated on the diagonal $*=s$).

The {\it Koszul construction} is defined by
$$K(\PP)_{(s)}=H_s(B_*(\PP)_{(s)}, \delta)= \text{ker} (\delta : B_s(\PP)_{(s)} \fdr B_{s-1}(\PP)_{(s)} ) .$$
From the definition, $K(\PP)_{(s)}$ is concentrated in degree $s$. We observe that the inclusion $K_d(\PP)_{(s)} \fdr B_d(\PP)_{(s)}$ is a morphism of complexes. The operad $\PP$ is Koszul if and only if the inclusion morphism $K(\PP) \fdr B(\PP)$ is a quasi-isomorphism.

\section{The language of trees}
Trees allow us to represent graphically the elements of the free operad and of the bar construction. The goal of this section is to define the conventions used throughout the article to describe the structure of a tree.
\vspace{0.3cm}

\subsection{Vertices and edges}\label{defarbres}
An {\it n-tree} is an abstract oriented tree together with one {\it outgoing edge} (the root of the tree) and $n$ {\it ingoing edges} (the entries of the tree) indexed by the set $\left\{1, \ldots, n\right\}$.
Formally, an $n$-tree $\tau$ is determined by a set of {\it vertices } $V(\tau)$ and by a set of {\it edges } $e \in E(\tau)$ oriented from a source $s(e) \in V(\tau) \coprod \left\{1, \ldots, n\right\}$ to a target $t(e) \in V(\tau) \coprod \left\{0\right\}$, with the following conditions :
\begin{enumerate}
	\item There is a unique edge $e \in  E(\tau)$ such that $t(e)=0$. We call this edge the {\it root}.
	\item For every vertex $v \in  V(\tau)$, there is a unique $e \in  E(\tau)$ such that $s(e) = v$.
	\item For every $i \in \left\{1, \ldots, n\right\}$, there is a unique $e$ such that $s(e)=i$. This edge is the $i$th entry of the tree.
	\item For every vertex $v$, there is a sequence of edges $e_1, \ldots, e_n$ such that $s(e_1)=v, t(e_i)=s(e_{i+1})$ for every $i \in \left[ 1,n-1 \right]$ and $t(e_n)= 0$.
\end{enumerate}

These conditions imply that the set $V(\tau) \coprod \left\{1, \ldots, n\right\}$ is equipped with a partial order so that $s(e)>t(e)$ for any edge $e$. The minimum of the order is $0$. There is an associated partial order on edges.

The set $E'(\tau)$ of {\it internal edges} is the set $E(\tau)$ of edges minus the ingoing edges and the outgoing edge.

We call a {\it leaf} the source of an ingoing edge. We draw
trees with leaves on top and the root at the bottom.

We say that a leaf $i$ is {\it linked to a vertex $v$} if there is a monotonic
path of edges between $i$ and $v$. We assume also that a leaf $i$ is linked to itself.

We define the entries of the vertex $v$ by $I_v = \left\{s(e),  e \in E(\tau) \textrm{ such that } t(e)=v \right\}$.

Then a tree structure is determined by a partition of the set $V(\tau) \coprod \left\{1, \ldots, n\right\}$ of the form $\coprod_{v \in V(\tau) \coprod \left\{0\right\}} I_v$.
\vspace{0.3cm}
 
\subsection{Tree isomorphisms}
An {\it isomorphism of  $n$-trees} $f : \tau \fdr \tau'$ is defined by two bijections
$$f_V : V(\tau) \fdr V(\tau') \textrm{ and } f_E : E(\tau) \fdr E(\tau') $$
which preserve the structure of the tree (the source and target of every edge).
We can extend $f_V$ by the identity on $\left\{1, \ldots, n\right\}$ to have the relation $I_{f_V (v)} = f_V (I_v)$ for every $v \in V(\tau) \coprod \left\{0\right\}$. The $n$-trees and their isomorphisms define a category.
\vspace{0.3cm}

\subsection{The $\Sigma_*$-category of trees}
Let $T(n)$ be the category defined by the $n$-trees and their isomorphisms. This category has a weight splitting : 
$$T(n) = \coprod_{r=0}^{\infty}T_{(r)}(n),$$
where $T_{(r)}(n)$ is the category formed by trees with $r$ vertices.

We can generalize the construction of $T(n)$ by indexing the entries of an $n$-tree by a set $I= \left\{i_1, \ldots, i_n\right\}$ of $n$ elements.
We obtain the category $T(I)$ of $I$-trees. A bijection $u : I \fdr I'$ induces a functor $u_* : T(I) \fdr T(I')$ such that 
$u_*(T_{(r)}(I)) \subset T_{(r)}(I')$.

A permutation $w \in \Sigma_n$ induces a functor from the category of $n$-trees to itself. Hence the symmetric group acts on $n$-trees.
\vspace{0.3cm}

\subsection{Subtrees}\label{sousarbre}

A {\it subtree} $\sigma$ of a tree $\tau$ is a tree determined by subsets 
$V(\sigma) \subset V(\tau)$ and $E(\sigma) \subset E(\tau)$ such that
$$v \in V(\sigma) \Longleftrightarrow \forall e \in E(\tau), (e \in E(\sigma) \Leftrightarrow s(e)=v \text{ or } t(e)=v).$$
The source and the target of an internal edge $e$ in $\sigma$ are
the source and the target of $e$ in $\tau$. A leaf $s(e)$ in $\sigma$ is
labelled by the minimum of the leaves which are linked to $s(e)$ in $\tau$.

Graphically, a subtree corresponds to a connected part of the graph of the tree.

The subtree $\sigma$ of a tree $\tau$ generated by an edge $e \in E'(\tau)$ is
the tree $\tau_e$ such that $V(\tau_e)=\{s(e),t(e)\}$ and
$E'(\tau_e)=\{e\}$. The ingoing edges relative to $s(e)$ and $t(e)$ are kept,
and the outgoing edge is linked to $t(e)$. The leaves are labelled as
specified above.
\vspace{0.3cm}

\subsection{The operad of trees}
One equips the sequence of categories $T(n)$ with the structure of an operad.
The partial composition product
$$ \circ_i : T_{(r)}(m) \x T_{(s)}(n) \fdr T_{(r+s)}(m+n-1)$$ 
is defined as follows : For $\sigma \in T_{(r)}(m)$ and $\tau \in T_{(s)}(n)$, the composite tree $\sigma \circ_i \tau$ is obtained by grafting the root of $\tau$ to the $i$th entry of $\sigma$ (\textit{cf}. figure 1 in the appendix at the end of the article).
\vspace{0.3cm}

\subsection{The module of treewise tensors}
Let $M$ be a $\Sigma_*$-module. A {\it module of treewise tensors} $\tau(M)$ is associated to any tree $\tau$.

Let $v$ be a vertex of $\tau$. Call $n_v$ the cardinal of $I_v$. Let $M(I_v)$ be the $\KK$-module generated by tensors
$f \t_{\Sigma_{n_v}} x_v$ where $x_v \in M(n_v)$ and $f$ is a bijection of the entries $\left\{1, \ldots, n \right\}$ to the entries of $x_v$.
One sets
$$\tau(M) = \bigotimes_{v \in V(\tau)} M(I_v).$$

Observe that this construction is functorial in $\tau$ : an isomorphism of trees $f : \tau \fdr \tau'$ induces a morphism $f_* : \tau(M) \fdr \tau'(M)$.

In practice, we see a treewise tensor as a tree with vertices labelled by elements of $M$, or equivalently a tensor product arranged on a tree.

Recall that a tree $\tau$ is called a corolla if it has only one vertex. For a corolla, we have an identification $\tau(M) \cong M(n)$ where $n$ is the number of entries of $\tau$.

\vspace{0.3cm}

\subsection{The free operad}\label{operadelibre}
The free operad has an explicit expansion so that

$$F(M)(n)= \bigoplus_{\tau \in T(n)} \tau(M) / \cong.$$

In $F(M)$, the relation $\cong$ identifies treewise tensors which correspond to each other by an isomorphism. Explicitly, for $x \in \tau(M)$ and $x' \in \tau'(M)$, we have $x' \cong x$ if and only if $x'=f_* x$ for an isomorphism $f : \tau \fdr \tau'$.

In this representation, the weight grading of the free operad defined in section \ref{poidsoperadelib} is given by the number of vertices of the tree.
\vspace{0.3cm}

\subsection{Construction without quotient}\label{sansquotient}

Throughout the paper, we work with trees (called {\it reduced}) verifying $I_v \neq \emptyset$ for every vertex $v$. A reduced tree has no automorphism except the identity. If $M(0)=0$, then the free operad involves only treewise tensors $x \in \tau(M)$ where $\tau$ is reduced. An operad is called {\it reduced} if it is spanned by treewise tensors on reduced trees.

\vspace{0.3cm}

We are going to use that a reduced tree has a canonical planar representation. This representation is determined by an ordering of the entries of each vertex $v$.

We determine an order on $I_{v}$ in the following way : 
\begin{enumerate}
	\item To every $v'$ in $I_{v}$, we associate the minimum of the leaves linked to $v'$.
	\item We place the vertices $v'$ and the leaves directly linked to $v$ from left to right above $v$ in ascending order.
\end{enumerate}

The order gives a bijection between $\{1,\ldots,n_v\}$ and the
entries of $v$. This bijection gives an isomorphim $M(I_v)\simeq M(n_v)$,
for each $v\in V(\tau)$. As a consequence, for the module of treewise
tensors $\tau(M)$, we obtain $\tau(M)\simeq\bigotimes_{v\in V(\tau)}
M(n_v)$.

To obtain a canonical representation of elements of the free operad, we
fix also a set $T'(n)$ of representatives of isomorphism classes of
$n$-trees. The expansion of the free operad gives then:
$$F(M)(n)\simeq\bigoplus_{\tau \in T'(n)} \tau(M)
\simeq\bigoplus_{\tau \in T'(n)} \bigl\{\bigotimes_{v\in V(\tau)}
M(n_v)\bigr\}.$$

\section{The Poincar\'e-Birkhoff-Witt criterion}

The aim of this section is to give the PBW criterion. We define the notion of a PBW basis for an operad, generalizing what Priddy did in the case of the algebras (\textit{cf}. \cite{Prid}).

\vspace{0.3cm}

\subsection{A basis of treewise tensors and of the free operad}\label{defbasetauM}
Let $M$ be a $\Sigma_*$-module, with an ordered basis $B^M$ (as a $\KK$-module) and such that $M(0)=0$. For every tree $\tau$, we define a {\it monomial basis} $B^{F(M)}_\tau$ of $\tau(M)$ in the following way. We use the planar representation of $\tau$, giving an isomorphism $\tau(M) \cong \bigotimes_v M(n_v)$. An element $\bigotimes_v m_v$ belongs to $B^{F(M)}_\tau$ if and only if each $m_v$ is in $B^M$. We set $B^{F(M)} = \coprod_\tau B^{F(M)}_\tau$.

A {\it pointed shuffle of a composition $\alpha \circ_i \beta$ } is a permutation preserving the order of the entries of each treewise tensor in the partial composition product and preserving the entry $i$. More explicitely, for $\alpha$ a treewise tensor with $s$ entries and $\beta$ a treewise tensor with $t$ entries, a permutation $w \in \Sigma_{s+t-1}$ is a pointed shuffle if the orders of the entries of $\alpha$ and of $\beta$ are the same as in the composition $w. \alpha \circ_i \beta$ and if the minimum of the entries of $\beta$ in the composition is $i$. This definition implies that the entries labelled $1$ to $i-1$ are not modified.

\begin{observ}\label{baseuniq}
The basis $B^{F(M)}$ is the only basis such that 

\begin{itemize}
	\item $B^{F(M)}_\tau=B^{M(n)}$ if $\tau$ is a corolla with  $n$ entries.
	\item For all $\alpha \in \sigma(M), \beta \in \tau(M)$ treewise tensors and $w$ pointed shuffle, we have : 
$$w. \alpha \circ_i \beta \in B^{F(M)}_{w. \sigma \circ_i \tau} \Leftrightarrow \alpha \in B^{F(M)}_\sigma \text { and } \beta \in B^{F(M)}_\tau.$$
\end{itemize}
\end{observ}
   
\subsection{Order on the basis of the treewise tensors}\label{defordre}
We are choosing an order on the monomial basis of $F(M)(r)$ for every $r$ in
$\NN$, verifying the compatibility condition :

For $\alpha, \alpha'$ with $m$ entries and $\beta,\beta'$ with $n$ entries, we have  \begin{equation} 
\left \{ \begin{array}{l} \alpha \leq \alpha' \\ \beta \leq \beta' \end{array} \right. \Rightarrow \forall i, w.\alpha \circ_i \beta \leq w.\alpha' \circ_i \beta', \  \forall w \text{ pointed shuffle}. 
\nonumber \end{equation}

\subsection{Example of a suitable order}\label{ordre}
Let $\alpha$ be a treewise tensor with $n$ entries.
We associate a sequence of $n$ words
$(\underline{a}_1,\underline{a}_2, \ldots,\underline{a}_n)$ to $\alpha$ in the
following way  : For all $i$, there
exists a unique monotonic path of vertices from the root to $i$, and $\underline{a}_i$
is the word composed (from left to right) of the labels of these vertices
(from bottom to top).

Recall that $M$ has an ordered basis. If $\underline a$ and $\underline b$ are two words, we first compare the length of the words ($\underline a < \underline b$ if $l(\underline a)<l(\underline b)$, where $l$ is the length)
and if they are equal, we compare them lexicographically (each letter being in
$M$). 

We can then compare two treewise tensors with the same number of entries $\alpha$ (associated to
$(\underline{a}_1,\underline{a}_2, \ldots,\underline{a}_n)$)  and $\beta$
(associated to  $(\underline{b}_1,\underline{b}_2, \ldots,\underline{b}_n)$), such that $\alpha \neq \beta$,
by comparing $\underline{a}_1$ with $\underline{b}_1$, then $\underline{a}_2$ with $\underline{b}_2$, etc. This defines the strict relation.

\begin{prop}
The order defined above verifies the compatibility condition of section \ref{defordre}.
\end{prop}

\begin{proof}
Let $\alpha$ and $\alpha'$ be treewise tensors with $n$ entries, such that $\alpha \leq \alpha'$.
Let $\beta$ and $\beta'$ be treewise tensors with $m$ entries, such that $\beta \leq \beta'$.

Let $(\underline{a}_1,\underline{a}_2, \ldots,\underline{a}_n)$, resp. $(\underline{a}'_1,\underline{a}'_2, \ldots,\underline{a}'_n)$, be the word sequence associated to $\alpha$, resp. $\alpha'$.
Let $(\underline{b}_1,\underline{b}_2, \ldots,\underline{b}_m)$, resp. $(\underline{b}'_1,\underline{b}'_2, \ldots,\underline{b}'_m)$, be the word sequence associated to $\beta$, resp. $\beta'$.

The word sequence associated to the composite $\alpha \circ_i \beta$  has the form
$(\underline{a}_1,\underline{a}_2,
\ldots,\underline{a}_i \underline{b}_1, \underline{a}_i \underline{b}_2,\ldots,
\underline{a}_i \underline {b}_m, \underline{a}_{i+1}, \ldots,
\underline{a}_n)$  where $\underline{a}_i\underline {b}_j$ is the concatenation of $\underline{a}_i$ 
and $\underline{b}_j$.
Similarly, the word sequence $(\underline{a}'_1,\underline{a}'_2,
\ldots,\underline{a}'_i \underline{b}'_1, \underline{a}'_i \underline{b}'_2,\ldots,
\underline{a}'_i\underline  b'_m, \underline{a}'_{i+1}, \ldots,
\underline{a}'_n)$ is associated to $\alpha' \circ_i \beta'$ . 

To begin with, note that the length of $\underline {a}_i \underline {b}_j$ is the sum of the
length of $\underline {a}_i$ and $\underline {b}_j$.

We compare $\alpha \circ_i \beta$ and $\alpha' \circ_i \beta'$ as they
have both $n+m-1$ entries.
As $\alpha \leq \alpha'$, we have  $(\underline a_1,\underline a_2,
\ldots,\underline a_{i-1}) \leq (\underline a'_1,\underline a'_2,
\ldots,\underline a'_{i-1})$. If the inequality is strict, we are done, as our order is lexicographical in the sequence.
If the inequality is an equality, we look at $\underline a_i$ and $\underline a'_i$. If $\underline a_i < \underline a'_i$, then $\underline a_i
\underline b_1 < \underline a'_i \underline b'_1$. If $\underline a_i=\underline a'_i$, comparing $\underline a_i \underline b_j$
with $\underline a'_i \underline b'_j$ is the same as comparing
$\underline b'_j$ with $\underline b'_j$ for all $j$. 
So $\underline a_i \underline b_1, \underline a_i \underline b_2,\ldots,
\underline a_i \underline  b_m \leq \underline a'_i \underline b'_1, \underline a'_i  \underline b'_2 ,\ldots,
\underline a'_i \underline  b'_m $. If the inequality is strict, we are done, else we have to look at the remainder of the sequence.
As $\alpha \leq \alpha'$ and $(\underline a_1 ,\underline a_2 ,
\ldots,\underline a_i ) = (\underline a'_1 ,\underline a'_2 ,
\ldots,\underline a'_i )$, we have  $(\underline a_{i+1} , \ldots,
\underline a_n ) \leq (\underline a'_{i+1} , \ldots,
\underline a'_n )$.

Finally we have  $\alpha \circ_i \beta
\leq \alpha' \circ_i \beta'$.

To show that $w.\alpha \circ_i \beta
\leq w.\alpha' \circ_i \beta'$ for all pointed shuffles $w$, we see how the pointed
shuffles act on the sequence of words associated to a composition of treewise
tensors.
The pointed shuffles will induce a shuffle (in the usual meaning) between the
set composed of the $\underline a_i \underline b_j$ for all $j$ and the set composed of
 the $\underline a_j $ for $j>i$. 

A shuffle preserves the order among each set it acts on and the order we have defined on the treewise tensors look
at the associated words recursively.
As a consequence, the order between $w.\alpha \circ_i \beta$ and $w.\alpha' \circ_i
\beta'$ will be the same as the one between $\alpha \circ_i \beta$ and $\alpha' \circ_i \beta'$.
\end{proof}

\subsection*{Remark} We call this order the {\it lexicographical order}. Another suitable order, the
{\it reverse-length lexicographical order}, can be defined in a similar way, 
If $\underline a$ and $\underline b$ are two words, we first compare the length of the words ($\underline a > \underline b$ if $l(\underline a)<l(\underline b)$, where $l$ is the length) and if they are equal, 
we compare them lexicographically (each letter being in $M$). 
The proof of the compatibility condition of section \ref{defordre} is the same.

\vspace{0.3cm}

\subsection{Restriction of a treewise tensor to a subtree}
Let $\alpha=\bigotimes_{v \in \tau} m_v$  be a treewise tensor.
The restriction of  $\alpha$ to a subtree $\sigma$ of $\tau$ is the tensor $\alpha_{|\sigma} =\bigotimes_{v \in V(\sigma)} m_v$ which gives an element of $\sigma(M)$.

We use this notion for a subtree $\sigma = \tau_e$ generated by an edge $e$ (defined in section \ref{sousarbre}).

\subsection{Poincar\'e-Birkhoff-Witt basis}

Let $\PP$ be a reduced operad defined by $F(M)/(\overline{R})$. 
A {\it PBW basis} for $\PP$ is a set $B^\PP \subset B^{F(M)}$ of elements representing a basis of the $\KK$-module $\PP$, containing $1$, $B^M$ and for every $\tau$ a subset $B_\tau^\PP$ of $B_\tau^{F(M)}$, verifying the conditions :

\begin{enumerate}
	\item For $\alpha \in B_\sigma^\PP$, $\beta \in B_\tau^\PP$ and $w$ a pointed shuffle, either $w. \alpha \circ_i \beta$ is in $B_{w. \sigma \circ_i \tau}^\PP$, or the elements of the basis $\gamma \in B^\PP$ which appear in the unique decomposition  $w. \alpha \circ_i \beta \equiv \Sigma_\gamma c_\gamma \gamma$,  verify $\gamma > w. \alpha \circ_i \beta$ in $F(M)$.
	\item A treewise tensor $\alpha$ is in $B_\tau^\PP$ if and only if for every internal edge $e$ of $\tau$, the restricted treewise tensor $\alpha_{|\tau_e}$ is in $B^\PP_{\tau_e}$.	
\end{enumerate}

\subsection*{Remark}
 This definition generalizes Priddy's definition for algebras (\textit{cf}. \cite{Prid}). Recall that an algebra $A$ is equivalent to an operad $\PP_A$ such that
$\PP_A(r) = \left \{ \begin{array}{rl} A, & \text{for $r=1$,} \\ 0, & \text{otherwise.} \end{array} \right. $

The algebra $A$ has a PBW basis in Priddy's sense
if and only if the operad $\PP_A$ has a PBW basis in our sense.

\begin{observ}
 Condition 1 is equivalent to condition 1' : 
\begin{itemize}
 \item[(1')] For $\alpha$ in $B^{F(M)}$, either $\alpha \in B^\PP$, or the elements of the basis $\gamma \in B^\PP$ which appear in the unique decomposition  $\alpha \equiv \Sigma_\gamma c_\gamma \gamma$,  verify $\gamma > \alpha$ in $F(M)$.

\end{itemize}

\end{observ}

\begin{proof}
 Condition 1' implies obviously condition 1. For the conserve direction, we use an induction on the number of vertices in $\alpha$ in $B^{F(M)}$ and observation \ref{baseuniq}.
\end{proof}

\begin{prop}\label{condquad}
Assume that $M$ is finitely generated.
If condition 1 is verified when $\alpha$ and $\beta$ are corollas, and condition 2 is verified, then condition 1 is true for all $\alpha$ and $\beta$.
\end{prop}

\begin{proof}
Equivalently, we can say : 

Let $M$ be finitely generated. If condition 1' is verified when $\alpha$ has only one internal edge and condition 2 is verified, then condition 1' is true for all $\alpha$.

We prove this equivalent proposition.

Let $\alpha$ be in $B_\tau^{F(M)} \setminus B_\tau^\PP$.
Condition 2 implies that there exists an internal edge $e$ such that $\alpha_{|\tau_e} \notin B^\PP_{\tau_e}$.
By condition 1', we can write $\alpha_{|\tau_e} \equiv \Sigma_\gamma c_\gamma \gamma$, where $\gamma > \alpha_{|\tau_e}$ in $F(M)$.
We replace $\alpha_{|\tau_e}$ by $\Sigma_\gamma c_\gamma \gamma$ in $\alpha$. This gives another representative of $\alpha \equiv \Sigma_{\gamma'} c_{\gamma'} \gamma'$ such that $\gamma' > \alpha$ (because the order is compatible with the partial composition product). If all $\gamma'$ are in $B^\PP$, we're done. Otherwise, we iterate the method.

We get others representative of $\alpha$ as sums of treewise tensors, each time strictly
larger. As the number of trees with a specified number of entries is finite and as
the basis of $M$ is also finite, then the number of treewise tensors with a
specified number of entries is also finite. So the process stops after a
finite number of steps, and the treewise tensors we get at the end are in $B^\PP$.
\end{proof}

\begin{theo}\label{thpbw}
A reduced operad which has a PBW basis is Koszul.
\end{theo}

The proof of this statement is achieved in the next section.

\section{Proof of the Poincar\'e-Birkhoff-Witt criterion}\label{demopbw}

To show this result, we describe more precisely $B(\PP)$ and a basis. Then we will use a filtration to study the homology of $E^0 B(\PP)(r)_\lambda $.

\subsection{Explicit description of $B(\PP)$}\label{bardetails}

By definition, $B(F(M))$ is equal to $\bigoplus_{\sigma} \sigma(F(M))$. Explicitly, a generator of $\sigma(F(M))$ corresponds to a tree $\sigma$ labelled with trees labelled by elements of $M$, that is a treewise tensor composed of treewise tensors on $M$. We can represent it by a large tree $\tau$ labelled by elements of $M$ and equipped with a splitting in subtrees $\tau_{comp}$, that we can see as connected components.
The $\tau_{comp}$ are separated by {\it cutting edges} which form a subset $D \subset E'(\tau)$. The union of the internal edges of the subtrees $\tau_{comp}$ form a set $S \subset E'(\tau)$ such that $S \coprod D = E'(\tau)$. We will work with $S$, the set of {\it marking edges}.

The marking edges $S$ determine the decomposition of a treewise tensor $\alpha$ into $\bigotimes \alpha_{comp}$ where $\alpha_{comp}=\alpha_{|\tau_{comp}}$ are the factors in $F(M)$.

So we identify an element of $B(F(M))$ to a pair $(\alpha, S)$, with $\alpha \in \tau(M)$ (\textit{cf}. figure 2).
$$B(F(M)) \cong \bigoplus_{\tau, S} (\tau(M), S).$$

\vspace{0.3cm}

We examine now the differential structure of $B(F(M))$.

The differential $\delta$ is given by
$$\delta(\alpha, S)= \sum_{e \in E'(\tau) - S} \pm (\alpha,S \coprod \{e\})$$
for $\alpha$ a treewise tensor associated to the tree $\tau$.

The operation $(\alpha, S) \mapsto (\alpha, S \coprod e)$ represents a partial composition at the edge $e$ for the element in $B(F(M))$ represented by $(\alpha,S)$.

Notice that the differential changes only the marking and not $\tau(M)$. Hence $\delta(\bigoplus_{S} (\tau(M), S)) \subset \bigoplus_{S} (\tau(M), S)$.

\subsection{Description and basis of $B(\PP)$}
First, let $B_\tau^{B(F(M))}$ be the natural basis of
treewise tensors on the tree $\tau$ labelled with elements of $B^{F(M)}$. Set also
$B^{B(F(M))} = \coprod_{\tau} B_\tau^{B(F(M))}$.

As $\PP = F(M)/(\overline{R})$, the reduced bar construction $B(\PP)$ is a quotient of $B(F(M))$. Two elements $(\alpha,S)$ and $(\alpha',S)$ are identified in  $B(\PP)$ if and only if $S=S'$ and every factor $\alpha_{comp}$ is identified to  $\alpha'_{comp}$ in $\PP$.

We define $B^{B(\PP)}$, a set of elements in $B(F(M))$ representing a basis of $B(\PP)$, starting from the base $B^\PP$ as follows : 
An element $(\beta, S)$ in $ (\tau(M), S)$ is in $B_\tau^{B(\PP)} \subset B_\tau^{B(F(M))}$ if every one of its factors  $\beta_{\tau_{comp}}$ lies in  $B_{\tau_{comp}}^{\PP}$. The element $\beta$ is an element in $B^{F(M)}_\tau$, the basis defined in section \ref{defbasetauM}.

\subsection*{Definition}An edge $e$ is said to be {\it admissible} if the restricted treewise tensor $\alpha_{|\tau_e}$ is in $B^{\PP}$. The set $Adm_{\alpha}$ is the set of the admissible edges of $\alpha$.

\begin{observ}
We have an equivalence
$$(\alpha,S) \in B^{B(\PP)} \Leftrightarrow S \subset Adm_{\alpha}.$$
\end{observ}

\subsection{Filtration of $B(\PP)$}

We consider a filtration $\displaystyle{B(\PP) =\bigcup_{\lambda \in I} B(\PP)_\lambda}$ where $I$ is a poset. This poset $I$ is defined by the basis of $F(M)$ and by the partial order specified in section  \ref{defordre}.

In practice, we forget the cutting and we use the partial order of the basis of $F(M)=\bigoplus_{\tau} \tau(M)$. Explicitely, an element $(\alpha,S) \in B(\PP)(r)$ is in  $B(\PP)(r)_\lambda$ if and only if $\alpha \geq \lambda$. Hence 
$$B(\PP)(r) = \bigcup_{\lambda \in I(r)} B(\PP)(r)_\lambda$$
where $I(r)$ is the monomial basis of $F(M)(r)$, a partially ordered set (with the order from section \ref{defordre}). 
\vspace{0.3cm}

Observe that $B(\PP)(r)_\lambda$ is a subcomplex of $B(\PP)(r)$. In fact, the differential $\delta$ corresponds to a partial composition product, modifying the cutting (that we forget in the filtration). Condition 1 of a PBW basis insures that an element is sent on the sum of larger or equal elements.
The differential $d^0$ induced by $\delta$ in the quotient preserves the factor $E^0 B(\PP)(r)_\lambda$, which is generated by the pairs $(\lambda, S)$ which belong to $B^{B(\PP)}$. 

Remark that $d^0_e : (\lambda,S) \mapsto (\lambda,S \coprod \{e\})$, so we can write $d^0= \sum_{e} \pm d^0_e$, taking the sum on the edges $e$ such that $d^0_e(\lambda)$ remains in the basis.

\begin{lemm}An edge $e$ is admissible if and only if $d^0_e (\lambda,S) \neq 0$. 
\end{lemm}

\begin{proof}
The differential $d^0_e$ transforms a non-marked admissible edge into a marked admissible edge, by condition 1 of a  PBW basis.

Conversely, if $d^0_e (\lambda,S) \neq 0$, then by condition 2 (converse direction), the edge is admissible. 
\end{proof}

\vspace{0.3cm}

\subsection{Homology of $E^0 B(\PP)(r)_\lambda$}\label{prophomo}

The quotient $E^0 B(\PP)(r)_\lambda $ is generated by the pairs $(\lambda, S)$
where $S$ ranges over the subsets of  $Adm_{\lambda}$, the set of admissible edges.
The differential $d^0_e$ sends $(\lambda,S)$ to $\displaystyle{(\lambda,\sum_{e \in Adm_\lambda - S} S \coprod \{e\})}$. 

This combinatorial complex is the dual of the oriented complex $C_ {\ast}(\Delta^{Adm_\lambda})_{+}$ of the simplex $\Delta^{Adm_\lambda}$ augmented over $\KK$, with the augmentation term added in $C_ {\ast}(\Delta^{Adm_\lambda})_{+}$. The inclusion of the summand spanned by $(\lambda, \emptyset)$ in $E^0 B(\PP)(r)_\lambda$ is dual to the augmentation $C_ {\ast}(\Delta^{Adm_\lambda})_{+} \rightarrow \KK$.

If $Adm_{\lambda}=\emptyset$, then the complex is reduced to a unique generator $(\lambda, \emptyset)$. Every component $\tau_{comp}$ is reduced to a vertex (we cut on all edges). This implies that the weight of $\lambda$ is equal to its degree.

If $Adm_{\lambda}$ is not empty, then the homology is zero.

\vspace{0.3cm}

We conclude from these assertions that $H_*E^0 B(\PP)(r)_\lambda =0$ if the weight is different from the degree.

The filtration is compatible with the weight. Hence the associated spectral sequence splits. We obtain $H_*B(\PP)=0$ when the weight is different from the degree. This result achieves the proof of theorem \ref{thpbw}.
\qed

\section{Result on the dual of a Poincar\'e-Birkhoff-Witt operad}

In this section we consider a reduced quadratic operad $\PP=F(M)/(R)$ such that $M$ is a finitely generated $\Sigma_*$-module. Recall that the Koszul construction $K(\PP)$ defined in section \ref{Koszul} is a cooperad, and its linear dual  $K(\PP)^\#$ is an operad.
The goal of this section is to prove the following result:

\begin{theo}
If $\PP$ is a PBW operad, then the dual operad $K(\PP)^\#$ is also a PBW operad.
\end{theo}

We determine a basis of $K(\PP)^\#$, and we prove it defines a PBW basis.

\vspace{0.3cm} 

\subsection{Basis of $K(\PP)^\#$}
First, in order to work with a filtration, we pick a total order which is a refinement of an order satisfying the condition of section \ref{defordre}. The only thing we use is that the subquotient $E^0 B(\PP)(r)_\lambda$ remains unchanged if we replace the partial order by any refinement.

We work here with a fixed number of entries $r$ and a fixed weight $n$. There is a finite number of trees with $n$ vertices and $r$ entries, so there is a finite basis of treewise tensors with $n$ vertices labelled by elements of $M$ and $r$ entries. 
This finite totally ordered set of treewise tensors can be written symbolically $\Lambda_{n,r}=\{ 0 < 1 < \ldots < \lambda < \lambda +1 < \ldots < \mu \}$.

Recall that an element $(\alpha,S) \in B(\PP)(r)$ is in  $B(\PP)(r)_\lambda$ if and only if $\alpha \geq \lambda$. In what follows, we write $F_\lambda = B(\PP)(r)_\lambda$ and $E^0_\lambda=E^0 B_{(n)}(\PP)(r)_\lambda $.

We have a finite filtration of $B_{(n)}(\PP)(r)$ :
$$B_{(n)}(\PP)(r) = F_0 \supseteq F_1 \supseteq \ldots \supseteq F_\lambda \supseteq F_{\lambda+1} \supseteq \ldots \supseteq F_\mu = E^0 _\mu.$$

\begin{sublemm}
For every $\lambda \in \Lambda_{n,r}$, the homology $H_{n-1} F_\lambda$ is $0$.
\end{sublemm}

\begin{proof}
We are using a decreasing induction.

For $\lambda=\mu$, we have $F_\mu = E^0_\mu$. We know that $H_{n-1} E^0 _\mu=0$ as the weight is different from the degree, so $H_{n-1} F_\mu$ is $0$.

Suppose the result true for $\lambda+1$. The long exact sequence in homology induced by $0 \fdr F_{\lambda+1} \fdr  F_{\lambda} \fdr E^0 _\lambda \fdr 0$
 gives 
$$ \ldots \fdr H_{n-1} F_{\lambda+1}  \fdr H_{n-1} F_{\lambda} \fdr H_{n-1} E^0 _\lambda \fdr \ldots .$$

The first term is $0$ by induction, and the third term is $0$ because the weight is different from the degree. So $H_{n-1} F_{\lambda}=0$.
\end{proof}

Another part of the long exact sequence gives the short exact sequence : 
$$ 0 \fdr H_{n} F_{\lambda+1}  \fdr H_{n} F_{\lambda} \fdr H_{n} E^0 _\lambda \fdr 0. $$
The first term $0$ comes from $H_{n+1} E^0_\lambda$ and the second one from $H_{n-1} F_\lambda$.

These short exact sequences can be put together in a diagram, where  vertical arrows are $\coker$'s :
\[\xymatrix{
H_{n} F_{\mu} \ar[r] & \ldots \ar[r] & H_{n} F_{\lambda+1}\ar[d]\ar[r] & H_{n} F_{\lambda}\ar[r]\ar[d]&  \ldots \ar[r] & H_{n} F_{0} .\\
                 &               & H_{n} E^0 _{\lambda+1} & H_{n} E^0 _{\lambda}& & & \\
}\]

Recall that $H_{n} F_\mu=0$ and  $H_{n} F_{0}=H_n (B_{(n)}(\PP)(r))$. We dualize this diagram, using that $H_n(C^\#)= H_n(C)^\#$, and we get the following diagram, where vertical arrows are Ker's :
\[\xymatrix{
0 & \ldots \ar[l] & H_{n} (F_{\lambda+1})^\#\ar[l] & H_{n} (F_{\lambda})^\#\ar[l]&  \ldots \ar[l] &\KK(\PP)^\#_{(n)}(r). \ar[l]\\
                 &               & H_{n} (E^0 _{\lambda+1})^\#\ar[u] & H_{n} (E^0 _{\lambda})^\#\ar[u]& & & \\
}\]

We showed in section \ref{prophomo} that 
$$H_{n} (E^0 _{\lambda}) = \left \{ \begin{array}{ll} \KK, & \textrm{ if } Adm_{\lambda}=\emptyset,  \\ 0, & \textrm{ otherwise.} \end{array} \right. $$

Thus we have $K(\PP)^\#_{(n)}(r) = \bigoplus_\lambda \KK$ where $Adm_{\lambda}=\emptyset$  with $\lambda \in \Lambda_{n,r}$.

Recall $K(\PP)^\#$ is a quotient of $F(\Sigma^{-1}M^\#)$. 
\begin{lemm}
 A basis of the $\KK$-module $K(\PP)^\#$ is represented by $\{ \lambda^\# \in F(\Sigma^{-1}(M^\#)) \ | \ Adm_{\lambda}=\emptyset\}$.
\end{lemm}

\begin{proof}
 This result is an obvious consequence of the description of $K(\PP)^\#_{(n)}(r)$ in the previous paragraph.
\end{proof}

These treewise tensors $\lambda$ are determined by the following property : the restricted treewise tensor induced by any edge $e$ is not in $B^\PP$.

\subsection{A PBW basis of $K(\PP)^\#$}

Ginzburg and Kapranov showed in \cite{GK} that $K(\PP)^\#= F(\Sigma^{-1}(M^\#))/(R')$, where $\overline{R'}$ is determined by the exact sequence 
$$ 0 \fdr \overline{R'} \fdr F_{(2)}(\Sigma^{-1}(M^\#)) \fdr K_{(2)}(\PP)^\# \fdr 0.$$

We have to determine $\overline {R'} \subset F_{(2)}(\Sigma^{-1}(M^\#))$ explicitely.

As $\overline R$ is characterized by $0 \fdr \overline R \fdr F_{(2)}(M) \fdr \PP_{(2)} \fdr 0$, we have dually 
$$ 0 \fdr \phi(\Sigma^{-2} \overline R^\bot) \fdr F_{(2)}(\Sigma^{-1}(M^\#)) \fdr K_{(2)}(\PP)^\# \fdr 0,$$
where $\phi$ is the isomorphism between $\Sigma^{-2} F_{(2)}(M^\#)$ and $F_{(2)}(\Sigma^{-1}(M^\#))$. We have to study $\phi(\Sigma^{-2} \overline R^\bot)$.

The main problem will be the suspensions which induce signs. Signs are induced by the classical commutation rule $g \otimes f = (-1)^{|f|.|g|}f \otimes g$. The suspension has degree $+1$.

Recall $F_{(2)}(M)$ is the set of treewise tensors with exactly one internal edge. Its basis $B^{F_{(2)}(M)}$ can be decomposed into $\displaystyle{\{\alpha_i\}_{i \in I} \coprod \{\alpha_j\}_{j \in J}}$ where $\forall i \in I, \alpha_i \notin B^\PP $ and $ \forall j \in J, \alpha_j \in B^\PP$.
The ideal generated by the relations is $\displaystyle{\overline R= \Span \{ \alpha_i-\sum_{ j \in J} c_{ij} \alpha_j ; i \in I \} } \subset F_{(2)}(M)$. A classic result of linear algebra gives $\displaystyle{\overline R^\bot= \Span \lbrace \alpha_j^\# + \sum_{ i \in I} c_{ij} \alpha_i^\# ; j \in J \rbrace} \subset F_{(2)}(M^\#)$.

For $x_1 \in M(n_1)$ and $x_2 \in M(n_2)$, the definition returns us the relation
$$\phi (\Sigma^{-2} w. x_1^\# \circ_i x_ 2^\#)= \epsilon(w) (-1)^{|x_1|} w. \Sigma^{-1} x_1^\# \circ_i \Sigma^{-1} x_2^\#,$$
where $\epsilon(w)$ denotes the signature of the permutation $w$.

As $\Sigma^{-2} \overline R^\bot \subset \Sigma^{-2} F_{(2)}(M^\#)$, we have $\phi(\Sigma^{-2} \overline R^\bot) \subset F_{(2)}(\Sigma^{-1}(M^\#))$. We have $\phi(\Sigma^{-2} \overline R^\bot)= \displaystyle{\Span \lbrace \phi(\Sigma^{-2} \alpha_j^\#) + \sum_{ i \in I} c_{ij} \phi(\Sigma^{-2} \alpha_i^\#) ; j \in J \rbrace }$, where we identify naturally $F_{(2)}(M^\#)$ and $F_{(2)}(M)^\#$.

\begin{theo}
Consider the set $B^\#$ formed by treewise tensors $\beta$ in $F(\Sigma^{-1}(M^\#))$, such that every subtensor generated by an internal edge of $\beta$ is in the set ${\phi(\Sigma^{-2} \alpha_i^\#), i \in I}$.
This set $B^\#$ forms a PBW basis of $K(\PP)^\#$ for the opposite order (denoted $<^\#$).
\end{theo}

Note $B^\#$ is uniquely determined by the $\phi(\Sigma^{-2} \alpha_i^\#), i \in I$. 

\begin{proof}
From the descriptions of $\overline R^\bot$ and of the basis of $K(\PP)^\#$, we observe that $B^\#$ is a basis of the module $K(\PP)^\#$. Also, condition 2 to be a PBW basis is true by definition.

Let us show condition 1. Here signs and suspensions do not interfere.
As the $\alpha_j, j \in J$ are the quadratic part of a PBW basis, we have $c_{ij} \neq 0$ for $\alpha_i<\alpha_j$. Hence we have $\alpha_i^\# >^\# \alpha_j^\#$ if $c_{ij} \neq 0$. As a consequence, condition 1 is verified for tensors with only one internal edge ({\it cf.} \ref{condquad}). As $M^\#$ is finitely generated, this implies condition 1 of a PBW basis.
\end{proof}

\subsection{Remark}
When the module $M$ is non-graded, we identify $M$ with a dg-object concentrated in degree $0$.

In the original construction by Ginzburg and Kapranov \cite{GK}, the Koszul dual $\PP^!$ is only defined for quadratic operads generated by binary operations. The original $\PP^!$ is an operadic suspension of $K(\PP)^\#$. The presentation by generators and relations has to be rewritten $\PP^! = F(M^\#)/(R'')$, and thus the operations in the dual $\PP^!$ remain in degree 0 if they were originally in degree 0 in $\PP$. This is not possible when the generators are not binary.

Because of signs, the orthogonal $(R'')$ is here $<\alpha_j^* + \sum_{ i \in I} c_{ij} \alpha_i^* ; j \in J>$, where $\alpha_k^*= \epsilon(w) (-1)^{|x_1|(n_2-1)} (-1)^{(i-1)(n_2-1)} w. x_1^\# \circ_i x_2^\# \in F_{(2)}(M^\#)$ if $\alpha_k = w. x_1 \circ_i x_2$. 
The operad $\PP^!$ has also a PBW basis, whose quadratic part is composed of the treewise tensors $\alpha_i^*, i \in I$.

In the case of operads generated by binary operations, we will work with $\PP^!$ rather than $K(\PP)^\#$, and determine the treewise tensors $\alpha_i^*, i \in I$ and the generating relations $R''$.

\section{Case of non-symmetric operads}
We obtain in a similar way a PBW criterion in the case of the non-symmetric operads.

\subsection{Non-symmetric operads and planar trees}
A non-symmetric operad is defined as an operad, but without the action of symmetric groups. For more details, we refer the reader to \cite{MSS}. We can represent compositions in a non-symmetric operad by planar trees.

The planar structure of a tree is determined by a total order on every set of entries $I_v$ for vertices $v \in V(\tau)$, as in the construction explained in section \ref{sansquotient}.
The planar structure induces a total order on the entries of the tree. When we work with non-symmetric operads, we always consider planar trees with a natural numerotation of the entries, the numeration preserving the order.

The non-symmetric free operad $F_{ns}(M)$ is associated to a non-symmetric
module, a sequence of modules  $M(n), n \in \NN$ without an action of
symmetric groups. We just replace abstract trees by planar trees in this construction.

\subsection{Order on the treewise tensors}
We define an order as in the symmetric case, the only difference being that we
forget pointed shuffles.

Let $M$ be a non-symmetric module, with an ordered basis $B^M$. For every
planar tree $\tau$, we have a natural {\it monomial basis}  $B^{F(M)}_\tau$ of
$\tau(M)$  : an element of this basis is the tree $\tau$ labelled with
elements of $B^M$.

We choose an order on the monomial basis of  $F_{ns}(M)(r)$ for every $r$ in
$\NN$, verifying the following condition :

For $\alpha, \alpha'$ with $m$ entries and $\beta,\beta'$ with $n$ entries,
we have
  \begin{equation} 
\left \{ \begin{array}{l} \alpha \leq \alpha' \\ \beta \leq \beta' \end{array} \right. \Rightarrow \forall i, \alpha \circ_i \beta \leq \alpha' \circ_i \beta'. 
\nonumber \end{equation}

\subsection{Non-symmetric PBW basis}

We define this notion as in the symmetric case, but without pointed shuffles.

Let $\PP$ be a non-symmetric operad, defined by  $F_{ns}(M)/(R)$. 

A {\it PBW basis} of $\PP$ is a set $B^\PP \subset
F_{ns}(M)$ of representatives of a base of the module $\PP$, containing $1$,
$B^M$ and for all  $\tau$ a subset $B_\tau^\PP$ of $B_\tau^{F(M)}$, verifying
the following properties :
\begin{enumerate}
\item For $\alpha \in B_\sigma^\PP$, $\beta \in B_\tau^\PP$, either $\alpha \circ_i \beta$ is in $B_{\sigma \circ_i \tau}^\PP$, or the elements of the basis $\gamma \in B^\PP$ which appear in the unique decomposition  $\alpha \circ_i \beta \equiv \Sigma_\gamma c_\gamma \gamma$,  verify $\gamma > \alpha \circ_i \beta$ in $F(M)$.
	\item A treewise tensor $\alpha$ is in $B_\tau^\PP$ if and only if for
      every internal edge $e$ of $\tau$, the restricted treewise tensor
      $\alpha_{|\tau_e}$ lies in  $B^\PP_{\tau_e}$.	
\end{enumerate}

\subsection{Symmetrization}
The forgetful functor from $\Sigma_*$-modules to sequences of (non-symmetric) modules
has a left adjoint  $\_ \t \Sigma_*$. If $\PP$ is a
non-symmetric operad, then the associated $\Sigma_*$-module  $\PP \t \Sigma_*$
has a natural operad structure.
For a free operad, we obtain $F_{ns}(M_{ns}) \t \Sigma_* = F(M_{ns} \t \Sigma_*)$.

We extend the order relation from $F_{ns}(M_{ns})$ to $F_{ns}(M_{ns}) \t
\Sigma_*$, setting : 
\begin{equation} 
\alpha \t \sigma \leq \alpha' \t \sigma' \text{ if } \left \{ \begin{array}{l} \sigma = \sigma' \\ \alpha \leq \alpha' \end{array} \right. . 
\nonumber \end{equation} 

We do not compare the elements if $\sigma \neq \sigma'$.

\begin{lemm}
A symmetric PBW basis  of $\PP$ is given by the orbits of a
non-symmetric PBW basis.
\end{lemm}

\begin{proof}
 Easy.
\end{proof}

As a corollary, we have :

\begin{theo}
A non-symmetric operad which has a non-symmetric PBW
basis is Koszul, and the non-symmetric dual operad has a non-symmetric PBW
basis, which can be explicitely determined from the other basis.
\qed
\end{theo}

\section{Examples}
We know that the following operads are Koszul : commutative $\mathcal C$, associative $\mathcal A$ and Lie $\mathcal Lie$ (\textit{cf}. Ginzburg and Kapranov \cite {GK}).
We  use our PBW criterion on these examples and on
some other operads. To simplify notations, we write sometimes relations with treewise
tensors in the operad, and sometimes in line in the associated algebra. We do not draw the root of the trees. In the examples with operads generated by binary operations, we will work with $\PP^!$, and determine the treewise tensors $\alpha_i^*, i \in I$ and the generating relations $R''$. Else we consider the dual $K(\PP)^\#$.

Recall that by condition 2, a treewise tensor is in the basis if and only if every subtensor generated by an edge is in the basis. As a consequence, we specify only the quadratic part of the basis to determine the basis completely. 

Verifications are omitted.

There are two main methods to find PBW bases :
\begin{itemize}
 \item We can start from a basis, and we need to find the an order on $M$ so it is a PBW basis (we have to check it verifies conditions 1 and 2).
 \item We can start from an ordered basis of $M$, which forces us the choice of the quadratic part (because of the relations). We then construct the set generated by this quadratic part (we are assured it verifies conditions 1 and 2) and we need to check if it is a basis (as a $\KK$-module).
\end{itemize}

\subsection{The associative operad}

In the binary case, the associative operad is generated by a single binary operation $\mu$, which
verifies the associativity relation $\mu(a,\mu(b,c))= \mu (\mu (a,b), c)$.

For the lexicographical order, the quadratic part of a non-symmetric PBW basis is given by $\begin{array}{c}
 \xymatrix@M=1pt@C=6pt@R=6pt{
1 \ar@{-}[dr] & & 2\ar@{-}[dl] & & 3 \ar@{-}[ddll] \\
 & *{\mu}\ar@{-}[dr]  & &  & \\
 & & *{\mu} & & 
}\end{array}$. The dual operad also has a non-symmetric PBW basis, whose quadratic part is given by $\begin{array}{c}
\xymatrix@M=1pt@C=6pt@R=6pt{
1 \ar@{-}[ddrr] & & 2\ar@{-}[dr] & & 3 \ar@{-}[dl] \\
 & & & *{\mu}\ar@{-}[dl]  & \\
 & & *{\mu} & & 
}\end{array}$. The relation is still the associativity relation, so the associative operad is self-dual.

\subsection{Generalization for higher associative operads}\label{pasbinaire}
It is possible to generalize the notion of associativity for operations of arity larger than $2$. For more details, we refer the reader to \cite{Gned}. The operads here are not generated by binary operations.

In the ternary case, one can define two types of associative operad. The totally associative operad satisfies 
$\mu(a, b, \mu(c,d,e))= \mu (a, \mu (b,c,d), e) = \mu(a,b,\mu(c,d,e))$, while the partially associative operad satisfies 
$\mu(a, b, \mu(c,d,e)) + \mu (a, \mu (b,c,d), e) + \mu(a,b,\mu(c,d,e))=0$. 

For the lexicographical order, the quadratic part of a non-symmetric PBW basis of the totally associative operad is 
$\begin{array}{c}
\xymatrix@M=1pt@C=6pt@R=6pt{
 1\ar@{-}[dr] & 2\ar@{-}[d]& 3\ar@{-}[dl] & 4\ar@{-}[ddl] &  5\ar@{-}[ddll] \\
 & *{\mu}\ar@{-}[dr] & &  & \\
 & & *{\mu} & & 
}\end{array}$. As a consequence, this operad is Koszul, and its dual $K(\PP)^\#$ is the partially associative operad where operations are in degree $1$, with the quadratic part of a PBW basis composed of 
$\begin{array}{c}
\xymatrix@M=1pt@C=6pt@R=6pt{
 1\ar@{-}[ddrr] & 2\ar@{-}[dr]& 3\ar@{-}[d] & 4\ar@{-}[dl] &  5\ar@{-}[ddll] \\
 &  &*{\mu}\ar@{-}[d] &  & \\
 & & *{\mu} & & 
}\end{array}$ and $\begin{array}{c}
\xymatrix@M=1pt@C=6pt@R=6pt{
 1\ar@{-}[ddrr] & 2\ar@{-}[ddr]& 3\ar@{-}[dr] & 4\ar@{-}[d] &  5\ar@{-}[dl] \\
 & & &*{\mu}\ar@{-}[dl]   & \\
 & & *{\mu} & & 
}\end{array}$ for the reverse-length lexicographical order. 

The same result can be shown for larger arities (with signs depending on the parity), {\it cf}. \cite{Gned}.

\subsection{The commutative and Lie operads}

The commutative operad is generated by a single binary operation $\mu$, which
verifies commutativity and associativity.
$$\mu(a,b)=\mu(b,a) \ \text{and} \ \mu(a,\mu(b,c))= \mu (\mu (a,b), c)$$

The relations in degree $2$ are
\begin{equation*}\begin{array}{c}
 $\xymatrix@M=1pt@C=6pt@R=6pt{
1 \ar@{-}[dr] & & 2\ar@{-}[dl] & & 3 \ar@{-}[ddll] \\
 & *{\mu}\ar@{-}[dr]  & &  & \\
 & & *{\mu} & & 
}$\end{array} = \begin{array}{c}
 $\xymatrix@M=1pt@C=6pt@R=6pt{
1 \ar@{-}[dr] & & 3\ar@{-}[dl] & & 2 \ar@{-}[ddll] \\
 & *{\mu}\ar@{-}[dr]  & &  & \\
 & & *{\mu} & & 
}$\end{array} = \begin{array}{c}
$\xymatrix@M=1pt@C=6pt@R=6pt{
1 \ar@{-}[ddrr] & & 2\ar@{-}[dr] & & 3 \ar@{-}[dl] \\
 & & & *{\mu}\ar@{-}[dl]  & \\
 & & *{\mu} & & 
}$\end{array}\end{equation*}

For the reverse-length lexicographical order, we get
\begin{equation*}\begin{array}{c}
 $\xymatrix@M=1pt@C=6pt@R=6pt{
1 \ar@{-}[dr] & & 2\ar@{-}[dl] & & 3 \ar@{-}[ddll] \\
 & *{\mu}\ar@{-}[dr]  & &  & \\
 & & *{\mu} & & 
}$\end{array} < \begin{array}{c}
 $\xymatrix@M=1pt@C=6pt@R=6pt{
1 \ar@{-}[dr] & & 3\ar@{-}[dl] & & 2 \ar@{-}[ddll] \\
 & *{\mu}\ar@{-}[dr]  & &  & \\
 & & *{\mu} & & 
}$\end{array} < \begin{array}{c}
$\xymatrix@M=1pt@C=6pt@R=6pt{
1 \ar@{-}[ddrr] & & 2\ar@{-}[dr] & & 3 \ar@{-}[dl] \\
 & & & *{\mu}\ar@{-}[dl]  & \\
 & & *{\mu} & & 
}$\end{array}, 
\end{equation*}

We check easily that the maximal element $\begin{array}{c}
$\xymatrix@M=1pt@C=6pt@R=6pt{
1 \ar@{-}[ddrr] & & 2\ar@{-}[dr] & & 3 \ar@{-}[dl] \\
 & & & *{\mu}\ar@{-}[dl]  & \\
 & & *{\mu} & & 
}$\end{array}$ in the quadratic relations is the quadratic part of
a PBW basis of the commutative operad (and as a consequence, this operad is
Koszul).

The dual operad is also Koszul, and the quadratic part of a PBW basis consists of two treewise tensors $\begin{array}{c}
 $\xymatrix@M=1pt@C=6pt@R=6pt{
1 \ar@{-}[dr] & & 2\ar@{-}[dl] & & 3 \ar@{-}[ddll] \\
 & *{[,]}\ar@{-}[dr]  & &  & \\
 & & *{[,]} & & 
 }$\end{array} \text{ and } \begin{array}{c}
 $\xymatrix@M=1pt@C=6pt@R=6pt{
1 \ar@{-}[dr] & & 3\ar@{-}[dl] & & 2 \ar@{-}[ddll] \\
 & *{[,]}\ar@{-}[dr]  & &  & \\
 & & *{[,]} & & 
 }$\end{array}$, where $[,]$ is the dual of $\mu$ and is anticommutative.

The relations in the dual operad is 
\begin{equation*}\begin{array}{c}
 $\xymatrix@M=1pt@C=6pt@R=6pt{
1 \ar@{-}[dr] & & 2\ar@{-}[dl] & & 3 \ar@{-}[ddll] \\
 & *{[,]}\ar@{-}[dr]  & &  & \\
 & & *{[,]} & & 
 }$\end{array} - \begin{array}{c}
 $\xymatrix@M=1pt@C=6pt@R=6pt{
1 \ar@{-}[dr] & & 3\ar@{-}[dl] & & 2 \ar@{-}[ddll] \\
 & *{[,]}\ar@{-}[dr]  & &  & \\
 & & *{[,]} & & 
 }$\end{array} = \begin{array}{c}
$\xymatrix@M=1pt@C=6pt@R=6pt{
1 \ar@{-}[ddrr] & & 2\ar@{-}[dr] & & 3 \ar@{-}[dl] \\
 & & & *{[,]}\ar@{-}[dl]  & \\
 & & *{[,]} & & 
}$\end{array}\end{equation*}

We recognize the Jacobi relation. So the operad $\mathcal Lie$ is the dual operad of $\mathcal C$, and as a consequence is Koszul.
Note that we retrieve the basis of Reutenauer in \cite[Section 5.6.2]{Reutenauer}.

\subsection{The Poisson operad}

The $\mathcal Poisson$ operad can be defined as $\mathcal C  \circ \mathcal
Lie$.

Explicitely, it is generated by $M= \KK . \bullet \oplus \KK [sgn] [,]$, with
the relations 
$$a \bullet (b \bullet c) = (a \bullet b) \bullet c \text { (Associativity)}$$
$$[[a,b],c]+[[b,c],a]+[[c,a],b]=0 \text{ (Jacobi)}$$
$$[a \bullet b,c]= a \bullet [b,c]+ b \bullet [a,c] \text{ (Poisson)}$$

We use the lexicographical order and we set $\bullet > [,]$.

We already know the quadratic part of a PBW basis for $\mathcal Lie$ and $\mathcal
Com$. After some calculations for the action of $\Sigma_3$ on the Poisson relation, 
we can determine the quadratic part of a PBW basis :
\begin{equation*}\begin{array}{c}
 $\xymatrix@M=1pt@C=6pt@R=6pt{
1 \ar@{-}[dr] & & 2\ar@{-}[dl] & & 3 \ar@{-}[ddll] \\
 & *{\bullet}\ar@{-}[dr]  & &  & \\
 & & *{\bullet} & & 
}$\end{array}, \begin{array}{c}
$\xymatrix@M=1pt@C=6pt@R=6pt{
1 \ar@{-}[dr] & & 2\ar@{-}[dl] & & 3 \ar@{-}[ddll] \\
 & *{[,]}\ar@{-}[dr]  & &  & \\
 & & *{[,]} & & 
 }$\end{array},  \begin{array}{c}
 $\xymatrix@M=1pt@C=6pt@R=6pt{
1 \ar@{-}[dr] & & 3\ar@{-}[dl] & & 2 \ar@{-}[ddll] \\
 & *{[,]}\ar@{-}[dr]  & &  & \\
 & & *{[,]} & & 
 }$\end{array}, 
 \end{equation*}
\begin{equation*}\begin{array}{c}
$\xymatrix@M=1pt@C=6pt@R=6pt{
1 \ar@{-}[dr] & & 2\ar@{-}[dl] & & 3 \ar@{-}[ddll] \\
 & *{\bullet}\ar@{-}[dr]  & &  & \\
 & & *{[,]} & & 
 }$\end{array},  \begin{array}{c}
 $\xymatrix@M=1pt@C=6pt@R=6pt{
1 \ar@{-}[dr] & & 3\ar@{-}[dl] & & 2 \ar@{-}[ddll] \\
 & *{[,]}\ar@{-}[dr]  & &  & \\
 & & *{\bullet} & & 
 }$\end{array}, \begin{array}{c}
$\xymatrix@M=1pt@C=6pt@R=6pt{
1 \ar@{-}[dr] & & 2\ar@{-}[dl] & & 3 \ar@{-}[ddll] \\
 & *{[,]}\ar@{-}[dr]  & &  & \\
 & & *{\bullet} & & 
}$\end{array}. \end{equation*}

So the Poisson operad is Koszul and the quadratic part of a PBW basis of its dual is :
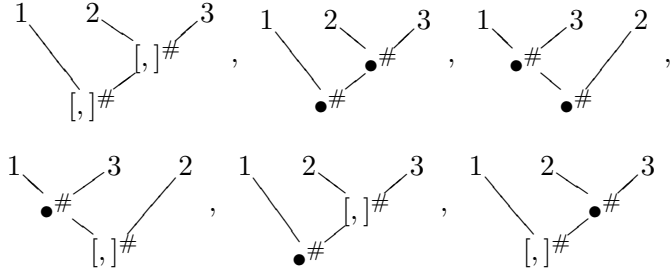
\begin{equation*}\begin{array}{c}
$\xymatrix@M=1pt@C=6pt@R=6pt{
1 \ar@{-}[ddrr] & & 2\ar@{-}[dr] & & 3 \ar@{-}[dl] \\
 & & & *{[,]^\#}\ar@{-}[dl]  & \\
 & & *{[,]^\#} & & 
}$\end{array}, \begin{array}{c}
$\xymatrix@M=1pt@C=6pt@R=6pt{
1 \ar@{-}[ddrr] & & 2\ar@{-}[dr] & & 3 \ar@{-}[dl] \\
 &   & & *{\bullet^\#}\ar@{-}[dl] & \\
 & & *{\bullet^\#} & & 
 }$\end{array},  \begin{array}{c}
 $\xymatrix@M=1pt@C=6pt@R=6pt{
1 \ar@{-}[dr] & & 3\ar@{-}[dl] & & 2 \ar@{-}[ddll] \\
 & *{\bullet^\#}\ar@{-}[dr]  & &  & \\
 & & *{\bullet^\#} & & 
 }$\end{array}, 
 \end{equation*}
 
 \begin{equation*}\begin{array}{c}
$\xymatrix@M=1pt@C=6pt@R=6pt{
1 \ar@{-}[dr] & & 3\ar@{-}[dl] & & 2 \ar@{-}[ddll] \\
 & *{\bullet^\#}\ar@{-}[dr] & &   & \\
 & & *{[,]^\#} & & 
}$\end{array},  \begin{array}{c}
$\xymatrix@M=1pt@C=6pt@R=6pt{
1 \ar@{-}[ddrr] & & 2\ar@{-}[dr] & & 3 \ar@{-}[dl] \\
 & & & *{[,]^\#}\ar@{-}[dl]  & \\
 & & *{\bullet^\#} & & 
}$\end{array}, \begin{array}{c}
$\xymatrix@M=1pt@C=6pt@R=6pt{
1 \ar@{-}[ddrr] & & 2\ar@{-}[dr] & & 3 \ar@{-}[dl] \\
 & & & *{\bullet^\#}\ar@{-}[dl]  & \\
 & & *{[,]^\#} & & 
}$\end{array}. \end{equation*}

The operation $\bullet^\#$ is anticommutative and satisfies the Jacobi relation, while the operation $[,]^\#$
is commutative and associative. The two operations together satisfy a Poisson relation. So we have retrieved that 
$\mathcal Poisson$ is self-dual, which was already proved by Markl using distributive laws in \cite{Markl}.

\subsection{The Perm and Prelie operads}

The $\mathcal Perm$ operad is defined by a single operation $\bullet$ satisfying :
$(x\bullet y) \bullet z = x \bullet (y \bullet z)= x \bullet (z \bullet y)$.

Let $\tau$ be the transposite $(12) \in \Sigma_2$.

For the lexicographical order and $\bullet > \tau \bullet$, a PBW basis is given by 
 \begin{equation*}\begin{array}{c}
 $
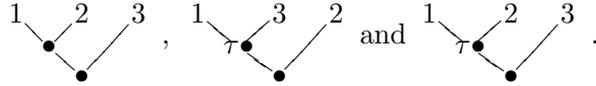
\xymatrix@M=1pt@C=6pt@R=6pt{
1 \ar@{-}[dr] & & 2\ar@{-}[dl] & & 3 \ar@{-}[ddll] \\
 & *{\bullet}\ar@{-}[dr]  & &  & \\
 & & *{\bullet} & & 
}$\end{array},  \begin{array}{c}
 $\xymatrix@M=1pt@C=6pt@R=6pt{
1 \ar@{-}[dr] & & 3\ar@{-}[dl] & & 2 \ar@{-}[ddll] \\
 & *{\tau \bullet}\ar@{-}[dr]  & &  & \\
 & & *{ \bullet} & & 
}$\end{array} \text{and} \begin{array}{c}
 $\xymatrix@M=1pt@C=6pt@R=6pt{
1 \ar@{-}[dr] & & 2\ar@{-}[dl] & & 3 \ar@{-}[ddll] \\
 & *{\tau \bullet}\ar@{-}[dr]  & &  & \\
 & & *{\bullet} & & 
}$\end{array}. \end{equation*}

The duals of the nine quadratic treewise tensors in the complementary are a PBW basis of the dual operad, 
and the relation ideal is generated by 
 \begin{equation*}\begin{array}{c}
 $\xymatrix@M=1pt@C=6pt@R=6pt{
1 \ar@{-}[dr] & & 2\ar@{-}[dl] & & 3 \ar@{-}[ddll] \\
 & *{\bullet^\#}\ar@{-}[dr]  & &  & \\
 & & *{\bullet^\#} & & 
}$\end{array}-  \begin{array}{c}
 $\xymatrix@M=1pt@C=6pt@R=6pt{
1 \ar@{-}[dr] & & 3\ar@{-}[dl] & & 2 \ar@{-}[ddll] \\
 & *{\bullet^\#}\ar@{-}[dr]  & &  & \\
 & & *{\bullet^\#} & & 
}$\end{array} - \begin{array}{c}
$\xymatrix@M=1pt@C=6pt@R=6pt{
1 \ar@{-}[ddrr] & & 2\ar@{-}[dr] & & 3 \ar@{-}[dl] \\
 & & & *{\bullet^\#}\ar@{-}[dl]  & \\
 & & *{\bullet^\#} & & 
}$\end{array} + \begin{array}{c}
$\xymatrix@M=1pt@C=6pt@R=6pt{
1 \ar@{-}[ddrr] & & 2\ar@{-}[dr] & & 3 \ar@{-}[dl] \\
 & & & *{\tau \bullet^\#}\ar@{-}[dl]  & \\
 & & *{\bullet^\#} & & 
}$\end{array}=0. \end{equation*}

This relation is known to define the $\mathcal Prelie$ operad. 
So we have shown that $\mathcal Prelie$ and $\mathcal Perm$ are Koszul and dual to each other.
This was already proved by Chapoton and Livernet in \cite{CL}.

\subsection{The  $m-Dend$ operad}
A $\KK$-vector space $V$ is an $m$-dendriform algebra if it is equipped with $m$ binary operations
$\bullet_1, \ldots, \bullet_m: V^{\otimes 2} \longrightarrow V$ verifying for all $x,y,z \in V,$ and for all $2 \leq i \leq m-1$,
the axioms
$$ (x\prec y)\prec z = x \prec (y \star z), \ \  \  \ (x \prec y)\bullet_i z =x\bullet_i (y \succ z) \ \forall \, 2 \leq i \leq m-1, $$ 
$$ (x \succ y)\prec z = x \succ (y \prec z), \ \  \  \ (x \succ y)\bullet_i z =x\succ (y \bullet_i z) \ \forall \, 2 \leq i \leq m-1, $$
$$(x \star y)\succ z =x\succ (y \succ z) , \ \  \  \ (x\bullet_i  y)\prec z = x \bullet_i (y \prec z) \ \forall \, 2 \leq i \leq m-1, $$
$$(x\bullet_i  y)\bullet_j z = x \bullet_i (y \bullet_j z) \ \ \forall \, 2 \leq i<j \leq m-1, $$
where $\bullet_1 := \succ$, $\bullet_m := \prec$ and $x \star y := x\prec y + x\succ y$.

We work with the associated operad, which was introduced by Leroux in \cite{Leroux}. He conjectured it was Koszul for $m>2$.
For $m=2$, the operad is the classical dendriform operad, introduced by Loday, and is Koszul \cite{Loday2}.

For the lexicographical order and $\bullet_i < \bullet_j$ if $i<j$, the quadratic part of a non-symmetric PBW basis is defined by all treewise tensors on 
$\begin{array}{c}
\xymatrix@M=1pt@C=6pt@R=6pt{
1 \ar@{-}[dr] & & 2\ar@{-}[dl] & & 3 \ar@{-}[ddll] \\
 & \ar@{-}[dr]  & &  & \\
 & &  & & 
} \end{array}$ and the following tensors : 
 \begin{equation*}\begin{array}{c}
$\xymatrix@M=1pt@C=6pt@R=6pt{
1 \ar@{-}[ddrr] & & 2\ar@{-}[dr] & & 3 \ar@{-}[dl] \\
 & & & *{\prec}\ar@{-}[dl]  & \\
 & & *{\prec} & & 
}$\end{array},  \begin{array}{c}
$\xymatrix@M=1pt@C=6pt@R=6pt{
1 \ar@{-}[ddrr] & & 2\ar@{-}[dr] & & 3 \ar@{-}[dl] \\
 & & & *{\bullet_i}\ar@{-}[dl]  & \\
 & & *{\prec} & & 
}$\end{array} \forall \, 2 \leq i \leq m-1 , \end{equation*}
\begin{equation*}\begin{array}{c}
$\xymatrix@M=1pt@C=6pt@R=6pt{
1 \ar@{-}[ddrr] & & 2\ar@{-}[dr] & & 3 \ar@{-}[dl] \\
 & & & *{\bullet_j}\ar@{-}[dl]  & \\
 & & *{\bullet_i} & & 
}$\end{array}\forall \, 2 \leq j \leq i \leq m-1. \end{equation*}

We have proved that the $m-Dend$ operad is Koszul, and so its dual $m-Tetra$ (calculated in \cite{Leroux}) is Koszul too.

\section*{Acknowledgements}

I am grateful to Benoit Fresse for many useful discussions
on this subject. I also would like to thank Muriel Livernet, Jean-Louis Loday, Martin Markl, Elisabeth Remm and Bruno Vallette
for their comments.

\newpage
\section*{Appendix}
{\bfseries Figure 1} : Example of a composition $\sigma \circ_i \tau$.

$\sigma$ : \[\xymatrix@M=1pt@C=6pt@R=6pt{
1 \ar@{-}[ddrr] & & 2\ar@{-}[dr] & & 3 \ar@{-}[dl] \\
 & & & v'_2\ar@{-}[dl]  & \\
 & & v'_1 \ar@{-}[d] & & \\
 & & 0 & &
}\]

$\tau$ : \[\xymatrix@M=1pt@C=6pt@R=6pt{
1 \ar@{-}[dr] & & 4\ar@{-}[dl] &  2 \ar@{-}[dr] & & 3\ar@{-}[dl] \\
  & v_2\ar@{-}[dr]& & & v_3\ar@{-}[dll] & \\
 & & v_1 \ar@{-}[d] & & & \\
 & & 0 & & & 
}\]

$\sigma \circ_1 \tau$ : \[\xymatrix@M=1pt@C=6pt@R=6pt{
1 \ar@{-}[dr] & 4\ar@{-}[d] & 2 \ar@{-}[d] & 3 \ar@{-}[dl]  & 5 \ar@{-}[d] & 6 \ar@{-}[dl]\\
 & v_2\ar@{-}[dr] & v_3 \ar@{-}[d] &  & v'_2 \ar@{-}[ddl] & \\
 & & v_1 \ar@{-}[dr] & & & \\
  & & & v'_1 \ar@{-}[d] & & \\
 & &  & 0 & & 
}\]

\vspace{1cm}
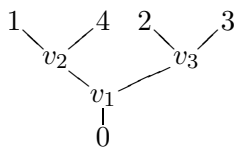
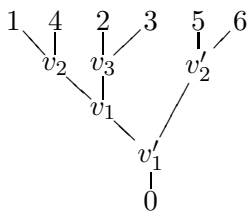
{\bfseries Figure 2} : The first treewise tensor $\lambda$ represents an element in $B(F(M))$, where $p_1$ and $p_2$ are elements in $F(M)$ and $i_1, \ldots, i_5$ a permutation of $1, \ldots, 5$. The edges in dots are the edges of the bar construction $B(F(M))$, the full edges are the edges of the free operad $F(M)$.

$\alpha$ : \[\xymatrix@M=1pt@C=6pt@R=6pt{
i_1 \ar@{.}[dr] & i_2\ar@{.}[d] & i_3 \ar@{.}[dl] & i_4 \ar@{.}[ddl]  & i_5 \ar@{.}[ddll] \\
 & p_2\ar@{.}[dr] & & \\
 & & p_1 \ar@{.}[d] &\\
 & & 0 & 
}\]

$p_1$ :\[\xymatrix@M=1pt@C=6pt@R=6pt{
j_1 \ar@{-}[dr] & & j_2\ar@{-}[dl] & & j_3 \ar@{-}[ddll] \\
 & x_2\ar@{-}[dr] & & & \\
 & & x_1 \ar@{-}[d] & & \\
 & & 0 & &
}\]

$p_2$ : \[\xymatrix@M=1pt@C=6pt@R=6pt{
k_1 \ar@{-}[ddrr] & & k_2\ar@{-}[dr] & & k_3 \ar@{-}[dl] \\
 & & & x'_2\ar@{-}[dl]  & \\
 & & x'_1 \ar@{-}[d] & & \\
 & & 0 & &
}\]

$\lambda=(\alpha,S)$ after substitution :
\[\xymatrix@M=1pt@C=6pt@R=6pt{
i_1 \ar@{.}[dr] & i_2 \ar@{.}[d] & i_3 \ar@{.}[dd] &   & i_4 \ar@{.} [d] & i_5 \ar@{.}[dl]\\
 & x_2\ar@{-}[dr] &  &  & x'_2 \ar@{-}[ddl] & \\
 & & x_1 \ar@{.}[dr] & & & \\
  & & & x'_1 \ar@{.}[d] & & \\
 & &  & 0 & & 
}\]

The set $D$ of cutting edges is reduced to the edge $\xymatrix@M=1pt@C=8pt@R=8pt{x_1 \ar@{.}[r] & x'_1}$.

The marked edges are the two full edges $x_2 - x_1$ and $x'_2 - x'_1$.

\end{document}